\documentclass[reqno]{amsart}
\usepackage[psamsfonts]{amssymb}
\usepackage{amsthm}
\usepackage{color}

\newtheorem{thm}{Theorem}
\newtheorem{lem}{Lemma}
\newtheorem{cor}{Corollary}

\theoremstyle{remark}

\theoremstyle{definition}
\newtheorem*{ack}{Acknowledgement}

\allowdisplaybreaks[4]

\begin{document}

\title[Differences of composition operators]{
Differences of composition operators from analytic Besov spaces into little Bloch type spaces
}

\author[A.K.~Sharma]{Ajay K.~Sharma}

\address{Ajay K. Sharma,
Department of Mathematics, 
Central University of Jammu, 
Bagla, 
Rahya-Suchani, Samba 181143, 
INDIA}

\email{aksju\_76@yahoo.com}

\author[S.~Ueki]{Sei-ichiro Ueki}

\address{Sei-ichiro Ueki, 
Department of Mathematics, 
Faculty of Science, 
Tokai University, 
Hiratsuka 259-1292, 
JAPAN}

\email{sei-ueki@tokai.ac.jp}

\begin{abstract}
The purpose of this paper is to describe the characterization for 
the compact difference of two composition operators 
acting between analytic Besov spaces and the weighted little Bloch type space over the unit disk. 
\end{abstract}

\keywords{Besov spaces, little Bloch spaces, composition operators.}

\subjclass[2010]{Primary 30H25; Secondary 30H30}

\maketitle

\section{Introduction}\label{sec:intr}
Let $\mathbb{D}$ denote the open unit disk in the complex plane $\mathbb{C}$ and 
$dA$ the normalized area measure on $\mathbb{D}$. 
Let $H(\mathbb{D})$ be the set of all analytic functions on $\mathbb{D}$. 
When $1 < p < \infty$, a function $f \in H(\mathbb{D})$ is said to be in the 
analytic Besov space $B^{p}$ if and only if 
\[
\int_{\mathbb{D}}|f^{\prime}(z)|^{p}(1-|z|^{2})^{p-2}dA(z) < \infty.
\]
And a norm $\|\cdot\|_{p}$ on $B^{p}$ is defined by 
\[
\|f\|_{p} = |f(0)| + 
\left[
\int_{\mathbb{D}}|f^{\prime}(z)|^{p}(1-|z|^{2})^{p-2}dA(z)
\right]^{1/p}.
\]
For the case $p =1$, the above integrable condition is satisfied by only constant functions. 
Thus the definition of  the space $B^{1}$ is complicated and 
there are several ways to define $B^{1}$. 
If $1<p<\infty$, it is well known that $f \in B^{p}$ is equivalent to 
\[
\int_{0}^{1}M_{p}^{p}(r,f^{\prime \prime})(1-r)^{p-1}dr < \infty.
\]
In the case $p=1$, the above condition becomes 
$\int_{0}^{1}M_{1}(r,f^{\prime \prime})dr < \infty$. 
Hence we can define the space $B^{1}$ by the condition
\[
\int_{\mathbb{D}}|f^{\prime \prime}(z)| dA(z) 
 < \infty. 
\]
For $w \in \mathbb{D}$, let $\alpha_{w}(z)$ be the conformal automorphism of $\mathbb{D}$ 
defined by 
$\alpha_{w}(z) = (w-z)/(1-\overline{w}z) ~ (z \in \mathbb{D})$. 
Each function $f \in B^{1}$ has an atomic decomposition, that is 
there exist sequences $\{c_{j}\} \in l^{1}$ and $\{w_{j}\} \subset \mathbb{D}$ such that 
\[
f(z) = c_{0} + \sum_{j =1}^{\infty}c_{j} \alpha_{w_{j}}(z) \qquad (z \in \mathbb{D}).
\]
By using this representation, a norm $\|\cdot\|_{1}$ on $B^{1}$ is defined by 
\[
\|f\|_{1} = \inf \sum_{j=0}^{\infty}|c_{j}|,
\]
where the infimum is taken over all $\{c_{j}\} \in l^{1}$ satisfy the above atomic decomposition 
for $f \in B^{1}$. 
It is known that $\|f\|_{1}$ is comparable to 
\[
|f(0)| + |f^{\prime}(0)| + \int_{\mathbb{D}}|f^{\prime \prime}(z)| dA(z) .
\]
For more details about analytic Besov spaces, 
we can refer to monographs \cite{WulZhu, KZhu}. 
Next we will introduce the weighted Bloch type space. 
Throughout this paper, let $\nu$ be a positive continuous radial function on $\mathbb{D}$. 
Here ``radial'' means that $\nu(z) = \nu(|z|)$ for $z \in \mathbb{D}$. 
The weighted Bloch type space $\mathcal{B}_{\nu}$ is the space of all $f \in H(\mathbb{D})$ 
which satisfy 
$\sup_{z \in \mathbb{D}}\nu(z)|f^{\prime}(z)| < \infty$, 
and the little Bloch type space $\mathcal{B}_{\nu, 0}$ consists of all $f \in \mathcal{B}_{\nu}$ 
satisfying $\nu(z)|f^{\prime}(z)| \to 0$ as $|z| \to 1^{-}$. 
It is easy to see that the space $\mathcal{B}_{\nu ,0}$ is a closed subspace in $\mathcal{B}_{\nu}$. 
These Bloch type spaces have been appeared in studies on composition or integral operators. 
For instance, S.~Stevi\'{c} and his collaborators have many studies about these operators 
acting on Bloch type spaces; see \cite{Ste09g,Ste09n,Ste11} and the related references therein. 

One of the major subjects in the field of analytic function spaces and operator theory 
is studies on composition operators. 
For an analytic self-map $\varphi$ of $\mathbb{D}$, 
the composition operator $C_{\varphi}$ is defined by $C_{\varphi}f = f \circ \varphi ~ (f \in H(\mathbb{D}))$. 
This composition operator has been studied extensively on various analytic function spaces. 
The aim of these studies is to explore the relation between 
operator-theoretic behaviors of $C_{\varphi}$ and 
function-theoretic properties of the map $\varphi$. 
Over the past few decades, a considerable number of studies have been conducted on 
the difference of composition operators on analytic function spaces. 
Shapiro and Sundberg \cite{ShSu} and MacCluer et al.~\cite{MOZ} studied 
a compact difference of composition operators on the Hardy spaces 
and topological structures of the space of composition operators. 
Moorhouse \cite{Moor} and Saukko \cite{Sauk11,Sauk12} have investigated 
characterizations for the compactness of the same operator on the weighted Bergman spaces, 
Hosokawa and Ohno \cite{HosOh} have considered it acting on the Bloch spaces. 
They used the pseudo-hyperbolic metric to give equivalent conditions for the compactness 
of the difference operator of compositions. 

Recently, motivated by these results, we have investigated this type operator from 
the analytic Besov space $B^{p}$ into the Bloch type space $\mathcal{B}_{\nu}$ in \cite{ShUe}. 
In that paper, we dealt with the case $C_{\varphi} - C_{\psi} \,:\, B^{p} \to \mathcal{B}_{\nu}$ only. 
Hence the purpose of this paper is to describe equivalent conditions for the compactness of 
$C_{\varphi} - C_{\psi} \,:\, B^{p} \to \mathcal{B}_{\nu,0}$. 
When we consider the case that the range of $C_{\varphi} - C_{\psi}$ is different from the domain of it, 
we have to take notice of the boundedness of it 
because a pair $\{\varphi , \psi\}$ does not always induce the bounded difference operator of compositions. 
In Section \ref{sec:bdd}, we will give characterizations for the boundedness of $C_{\varphi} - C_{\psi}$ which 
the range is $\mathcal{B}_{\nu, 0}$. 
By applying this result for the boundedness, 
we will describe characterizations for the compactness of $C_{\varphi} - C_{\psi}$. 
Section \ref{sec:comp} is devoted to explain the details of them. 

Throughout this paper, the notation $A \lesssim B$ means that 
there exists a positive constant $C$ such that $A \le CB$. 
Of course, the constant $C$ is independent of a function $f$, 
a point $z \in \mathbb{D}$ and related parameters $\{t, r\}$. 
Moreover, if both $A \lesssim B$ and $B \lesssim A$ hold, 
then one says that $A \approx B$.
 
\section{Preliminaries}\label{sec:prelimi}

We will need the following results in Section \ref{sec:bdd} and \ref{sec:comp}.

\begin{lem}\label{lem:growth}
Let $1 \le p < \infty$ and $f \in B^{p}$. 
Then 
\[
|f^{\prime}(z)| \lesssim \frac{\|f\|_{p}}{1-|z|^{2}}
\]
for all $z \in \mathbb{D}$.
\end{lem}

\begin{proof}
We have to consider the two cases $p \neq 1$ and $p =1$. 
For the case $p \neq 1$, by the definition of the space $B^{p}$, 
$f \in B^{p}$ if and only if $f^{\prime}$ belongs to the classical weighted Bergman space $L_{a}^{p}(dA_{p-2})$. 
Hence $f^{\prime}$ has the following point evaluation estimate: 
\[
|f^{\prime}(z)| \le \dfrac{\|f^{\prime}\|_{L_{a}^{p}(dA_{p-2})}}{1-|z|^{2}}
\]
for all $z \in \mathbb{D}$. Since $\|f^{\prime}\|_{L_{a}^{p}(dA_{p-2})} \le \|f\|_{p}$, 
we obtain the desired estimate. 
To prove the case $p=1$, we use the atomic decomposition of $f \in B^{1}$. 
If $f \in B^{1}$, we can choose sequences $\{c_{j}\} \in l^{1}$ and $\{w_{j}\} \subset \mathbb{D}$ 
such that $f = c_{0} + \sum c_{j} {\alpha}_{w_{j}}$. 
Thus we have $|f(z)| \lesssim \sum |c_{j}|$ for all $z \in \mathbb{D}$. 
By taking the infimum with respect to all such representation of $f$, 
we obtain $|f(z)| \lesssim \|f\|_{1}$ for all $z \in \mathbb{D}$. 
An application of Cauchy's estimate to $f^{\prime}$ on the circle with
center at $z$ and radius $(1-|z|)/2$ shows 
$|f^{\prime}(z)|\lesssim \|f\|_{1}/(1-|z|^{2})$ for all $z \in \mathbb{D}$.
\end{proof}

\begin{lem}\label{lem:lip}
Let $1 \le p < \infty$ and $f \in B^{p}$. 
Then 
\[
|(1-|z|^{2})f^{\prime}(z) - (1-|w|^{2})f^{\prime}(w)| \lesssim \|f\|_{p}\rho(z,w)
\]
for all $\{z, w \} \subset \mathbb{D}$.
\end{lem}

\begin{proof}
In \cite[Proposition 2.2]{HosOh}, Hosokawa and Ohno proved that  
\[
|(1-|z|^{2})f^{\prime}(z) - (1-|w|^{2})f^{\prime}(w)| \lesssim \rho(z,w) \sup_{\zeta \in \mathbb{D}}(1-|\zeta|^{2})|f^{\prime}(\zeta)|
\]
for $f$ belongs to the Bloch space $\mathcal{B}$ and $\{z, w \} \subset \mathbb{D}$. 
Since Lemma \ref{lem:growth} imply that $\mathcal{B} \subset B^{p} ~ (1 \le p < \infty)$ and 
$\sup_{\zeta \in \mathbb{D}}(1-|\zeta|^{2})|f^{\prime}(\zeta)| \lesssim \|f\|_{p}$, 
the desired estimate can be verified by the above estimate. 
\end{proof}

A compact subset of $\mathcal{B}_{\nu ,0}$ can be characterized as following. 
The same result for the usual little Bloch space $\mathcal{B}_{0}$ 
was proved by Madigan and Matheson \cite{MaMa}. 
By a slightly modification of their proof, we can prove the following lemma. 

\begin{lem}\label{lem:bddcomp}
A closed subset $L$ in $\mathcal{B}_{\nu , 0}$ is compact if and only if 
it is a bounded subset in $\mathcal{B}_{\nu}$ and satisfies 
\[
\lim_{|z| \to 1^{-}} \sup_{f \in L}\nu(z)|f^{\prime}(z)| = 0. 
\]
\end{lem}

The following result is appeared in our previous work \cite{ShUe}. 
We will need it in the argument of the compactness in Section \ref{sec:comp}.
\begin{thm}\label{thm:blochbdd}
Let $1 \le p < \infty$ and $\{\varphi , \psi\}$ a pair of analytic self-maps of $\mathbb{D}$. 
Then the following statements are equivalent:
\begin{enumerate}
\item[(i)] $C_{\varphi} - C_{\psi} \,:\, B^{p} \to \mathcal{B}_{\nu}$ is bounded, 
\item[(ii)] $\varphi$ and $\psi$ satisfy the following two conditions:
\[
\sup_{z \in \mathbb{D}} \frac{\nu(z)|{\varphi}^{\prime}(z)|}{1-|\varphi(z)|^{2}}\rho(\varphi(z), \psi(z))<\infty
\]
and 
\[
\sup_{z \in \mathbb{D}}\left| \frac{\nu(z){\varphi}^{\prime}(z)}{1-|\varphi(z)|^{2}}
 - \frac{\nu(z){\psi}^{\prime}(z)}{1-|\psi(z)|^{2}} \right| < \infty ,
 \]
\item[(iii)] $\varphi$ and $\psi$ satisfy the following two conditions:
\[
\sup_{z \in \mathbb{D}} \frac{\nu(z)|{\psi}^{\prime}(z)|}{1-|\psi(z)|^{2}}\rho(\varphi(z), \psi(z))<\infty
\]
and 
\[
\sup_{z \in \mathbb{D}}\left| \frac{\nu(z){\varphi}^{\prime}(z)}{1-|\varphi(z)|^{2}}
 - \frac{\nu(z){\psi}^{\prime}(z)}{1-|\psi(z)|^{2}} \right| < \infty .
 \]
\end{enumerate}
\end{thm}

\section{Boundedness of $C_{\varphi} - C_{\psi}$}\label{sec:bdd}

Before considering the compactness of $C_{\varphi} - C_{\psi}$, 
we have to mention the boundedness of it. 
The following Theorem \ref{thm:general_bdd} can be found in \cite[Theorem 3.4]{HosOh}. 
They proved the result for the case that $C_{\varphi}-C_{\psi}$ is acting on the little Bloch space $\mathcal{B}_{0}$. 
Under the assumption on the boundedness of $C_{\varphi} - C_{\psi}$ and the density of the polynomial set in the domain space, 
we can generalize their result as following. 

\begin{thm}\label{thm:general_bdd}
Let $X$ be a Banach space of analytic functions over $\mathbb{D}$
which the polynomial set is dense in $X$. 
For each pair $\{\varphi , \psi\}$ of analytic self-maps of $\mathbb{D}$ 
with $C_{\varphi} - C_{\psi} \, :\, X \to \mathcal{B}_{\nu}$ is bounded, 
the following conditions are equivalent:
\begin{enumerate}
\item[(a)] $C_{\varphi} - C_{\psi} \, :\, X \to \mathcal{B}_{\nu, 0}$ is bounded, 
\item[(b)] $\varphi - \psi \in \mathcal{B}_{\nu, 0}$ and ${\varphi}^{2} - {\psi}^{2} \in \mathcal{B}_{\nu, 0}$,
\item[(c)] $\varphi - \psi \in \mathcal{B}_{\nu , 0}$ and 
\[
\lim_{|z| \to 1^{-}}\nu(z)|\varphi(z) - \psi(z)|\max\{|{\varphi}^{\prime}(z)|, \, |{\psi}^{\prime}(z)|\} = 0.
\]
\end{enumerate}
\end{thm}

\begin{proof}
The direction (a) $\Rightarrow$ (b) is verified by test functions $p_{1}(z) = z$ and $p_{2}(z) = z^{2}$ easily. 
Hence it is enough to prove directions (b) $\Rightarrow$ (c) and (c) $\Rightarrow$ (a). 
Now we will prove (b) $\Rightarrow$ (c). 
Since $\varphi - \psi \in \mathcal{B}_{\nu , 0}$ implies $\nu(z)|{\varphi}^{\prime}(z) - {\psi}^{\prime}(z)| \to 0$ as 
$|z| \to 1^{-}$ and ${\varphi}^{2} - {\psi}^{2} \in \mathcal{B}_{\nu , 0}$ implies 
$\nu(z)|\varphi(z){\varphi}^{\prime}(z) - \psi(z){\psi}^{\prime}(z)| \to 0$ as $|z| \to 1^{-}$, 
we obtain that 
\begin{align*}
& \nu(z)|\varphi(z) - \psi(z)||{\varphi}^{\prime}(z)| \\
& \le \nu(z)|\varphi(z){\varphi}^{\prime}(z) - \psi(z){\psi}^{\prime}(z)| + \nu(z)|{\varphi}^{\prime}(z) - {\psi}^{\prime}(z)||\psi(z)| \\
& \le \nu(z)|\varphi(z){\varphi}^{\prime}(z) - \psi(z){\psi}^{\prime}(z)| + \nu(z)|{\varphi}^{\prime}(z) - {\psi}^{\prime}(z)|
\to 0,
\end{align*}
as $|z| \to 1^{-}$. We also have $\nu(z)|\varphi(z) - \psi(z)||{\psi}^{\prime}(z)| \to 0$ as $|z| \to 1^{-}$, 
and so the condition (c) is true. 
In order to prove (c) $\Rightarrow$ (a), we assume (c). 
For each $n \ge 1$, we put $p_{n}(z) = z^{n}$. 
Then 
\[
(C_{\varphi}-C_{\psi})p_{n}(z) = {\varphi}^{n}(z) - {\psi}^{n}(z) = (\varphi(z) - \psi(z))
\sum_{k=0}^{n-1}{\varphi}^{n-1-k}(z){\psi}^{k}(z).
\]
We will claim that $(C_{\varphi}-C_{\psi})p_{n} \in \mathcal{B}_{\nu, 0}$. 
Since 
\begin{align*}
& \left( \sum_{k=0}^{n-1}{\varphi}^{n-1-k}(z){\psi}^{k}(z) \right)^{\prime} \\
& = 
(n-1){\varphi}^{n-2}(z) {\varphi}^{\prime}(z) + 
(n-2){\varphi}^{n-3}(z){\varphi}^{\prime}(z)\psi(z) + {\varphi}^{n-2}(z){\psi}^{\prime}(z) \\
& \qquad 
+ \dotsm + 
{\varphi}^{\prime}(z){\psi}^{n-2}(z) + (n-2)\varphi(z){\psi}^{n-3}(z){\psi}^{\prime}(z) + (n-1){\psi}^{n-2}(z){\psi}^{\prime}(z),
\end{align*}
we have that 
\[
\left|\left( \sum_{k=0}^{n-1}{\varphi}^{n-1-k}(z){\psi}^{k}(z) \right)^{\prime}\right|
\le \frac{n(n-1)}{2}(|{\varphi}^{\prime}(z)|+|{\psi}^{\prime}(z)|).
\]
Hence this inequality gives that 
\begin{align*}
& \nu(z)|((C_{\varphi} - C_{\psi})p_{n})^{\prime}(z)| \\
& \le 
n\nu(z)|{\varphi}^{\prime}(z) - {\psi}^{\prime}(z)| 
+ 
\frac{n(n-1)}{2}\nu(z)|\varphi(z) - \psi(z)|(|{\varphi}^{\prime}(z)|+|{\psi}^{\prime}(z)|).
\end{align*}
Combining this estimate with the condition (c), we see that $(C_{\varphi}-C_{\psi})p_{n} \in \mathcal{B}_{\nu, 0}$, 
and so $(C_{\varphi}-C_{\psi})p \in \mathcal{B}_{\nu, 0}$ for all analytic polynomial $p$. 
Since the polynomial set is dense in X, $C_{\varphi} - C_{\psi} \, :\, X \to \mathcal{B}_{\nu}$ is bounded 
and $\mathcal{B}_{\nu, 0}$ is closed in $\mathcal{B}_{\nu}$, we also see 
$(C_{\varphi}- C_{\psi})f \in \mathcal{B}_{\nu, 0}$ for $f \in X$. 
This implies the boundedness of $C_{\varphi} - C_{\psi} \, :\, X \to \mathcal{B}_{\nu, 0}$. 
\end{proof}

\begin{lem}\label{lem:dense}
For $1 \le p < \infty$, the polynomial set is dense in $B^{p}$. 
\end{lem}

\begin{proof}
Each dilated function $f_{r}$ is analytic in the closed unit disk $\overline{\mathbb{D}}$, 
and so it belongs to $B^{p}$. 
Since $f_{r}$ is approximated by polynomials in $B^{p}$, 
it is enough to prove that every $f \in B^{p}$ satisfies 
$\|f - f_{r}\|_{p} \to 0 $ as $r \to 1^{-}$. 
First we will consider the case $p >1$. 
Fix $\varepsilon >0$. 
Since $f \in B^{p}$, there exists an $R \in (0,1)$ such that 
$\int_{\mathbb{D} \setminus R\overline{\mathbb{D}}} |f^{\prime}(z)|^{p}(1-|z|^{2})^{p-2}dA(z) < \varepsilon$. 
Noting that $|f^{\prime}|^{p}$ is subharmonic in $\mathbb{D}$, 
we have
\begin{align*}
\int_{\mathbb{D} \setminus R\overline{\mathbb{D}}} |f^{\prime}(rz)|^{p}(1-|z|^{2})^{p-2}dA(z) 
& = 
2\int_{R}^{1}t(1-t^{2})^{p-2}dt
\int_{0}^{2\pi}|f^{\prime}(rte^{i\theta})|^{p}\dfrac{d\theta}{2\pi}
\\
& \le 
2\int_{R}^{1}t(1-t^{2})^{p-2}dt
\int_{0}^{2\pi}|f^{\prime}(te^{i\theta})|^{p}\dfrac{d\theta}{2\pi}
\\
& = 
\int_{\mathbb{D} \setminus R\overline{\mathbb{D}}} |f^{\prime}(z)|^{p}(1-|z|^{2})^{p-2}dA(z) < \varepsilon
\end{align*}
for any $r \in (0,1)$. 
Hence we obtain 
\begin{align*}
\|f - f_{r}\|_{p}^{p}
& = \int_{\mathbb{D}}|f^{\prime}(z) - r f^{\prime}(rz)|^{p}(1-|z|^{2})^{p-2}dA(z)
\\
& \lesssim 
(1-r)^{p}\int_{\mathbb{D}}|f^{\prime}(rz)|^{p}(1-|z|^{2})^{p-2}dA(z)
\\
& \qquad \qquad +
\int_{\mathbb{D}}|f^{\prime}(z) - f^{\prime}(rz)|^{p}(1-|z|^{2})^{p-2}dA(z)
\\
& \lesssim 
(1-r)^{p}\|f\|_{p}^{p} + \varepsilon 
+ \int_{R\overline{\mathbb{D}}}|f^{\prime}(z) - f^{\prime}(rz)|^{p}(1-|z|^{2})^{p-2}dA(z).
\end{align*}
Since $f^{\prime}$ is uniform continuous on $R\overline{\mathbb{D}}$, 
it follows from this estimate that $\|f - f_{r}\|_{p} \to 0$ as $r \to 1^{-}$. 
For the case $p=1$, we obtain
\begin{align*}
\|f - f_{r}\|_{1} 
& \lesssim 
\int_{\mathbb{D}}|f^{\prime \prime}(z) - r^{2}f^{\prime \prime}(rz)|dA(z)
\\
& \le 
(1-r^{2})\int_{\mathbb{D}}|f^{\prime \prime}(rz)| dA(z) 
+ 
\int_{\mathbb{D}}|f^{\prime \prime}(rz) - f^{\prime \prime}(z)|dA(z).
\end{align*}
By the same argument as in the case $p> 1$, 
these inequalities also show that $\|f - f_{r}\|_{1} \to 0$ as $r \to 1^{-}$. 
\end{proof}

\begin{cor}\label{cor:bdd}
Let $1 \le p < \infty$ and $\{\varphi, \psi\}$ a pair of analytic self-maps of $\mathbb{D}$ 
which induces the bounded operator $C_{\varphi} - C_{\psi} \, : \, B^{p} \to \mathcal{B}_{\nu}$. 
Then the following conditions are equivalent:
\begin{enumerate}
\item[(a)] $C_{\varphi} - C_{\psi} \, :\, B^{p} \to \mathcal{B}_{\nu, 0}$ is bounded, 
\item[(b)] $\varphi - \psi \in \mathcal{B}_{\nu, 0}$ and ${\varphi}^{2} - {\psi}^{2} \in \mathcal{B}_{\nu, 0}$,
\item[(c)] $\varphi - \psi \in \mathcal{B}_{\nu , 0}$ and 
\[
\lim_{|z| \to 1^{-}}\nu(z)|\varphi(z) - \psi(z)|\max\{|{\varphi}^{\prime}(z)|, \, |{\psi}^{\prime}(z)|\} = 0.
\]
\end{enumerate}
\end{cor}

\section{Compactness of $C_{\varphi} - C_{\psi}$}\label{sec:comp}

\begin{thm}\label{thm:cmpt}
Let $1 \le p < \infty$. 
For each $\{\varphi, \psi\}$ a pair of analytic self-maps of $\mathbb{D}$,  
$C_{\varphi} - C_{\psi} \,:\, B^{p} \to \mathcal{B}_{\nu , 0}$ is compact 
if and only if $\varphi$ and $\psi$ satisfy the following two condtions:
\begin{enumerate}
\item[(a)] $\displaystyle 
\lim_{|z| \to 1^{-}}
\max\left\{\dfrac{|{\varphi}^{\prime}(z)|}{1-|\varphi(z)|^{2}}, \dfrac{|{\psi}^{\prime}(z)|}{1-|\psi(z)|^{2}}\right\} 
\nu(z)\rho(\varphi(z), \psi(z)) = 0
$,
\item[(b)] $\displaystyle 
\lim_{|z| \to 1^{-}}
\left|
\dfrac{\nu(z){\varphi}^{\prime}(z)}{1-|\varphi(z)|^{2}} - \dfrac{\nu(z){\psi}^{\prime}(z)}{1-|\psi(z)|^{2}}
\right| = 0
$.
\end{enumerate}
\end{thm}

\begin{proof}
Now we assume that conditions (a) and (b) are true. 
Let $K =\{ f \in B^{p} \,:\, \|f\|_{p} \le 1\}$ closed unit ball in $B^{p}$. 
In order to prove the compactness of $C_{\varphi} - C_{\psi} \,:\, B^{p} \to \mathcal{B}_{\nu ,0}$, 
by Lemma \ref{lem:bddcomp}, we may prove that $C_{\varphi} - C_{\psi} \,:\, B^{p} \to \mathcal{B}_{\nu ,0}$ 
is bounded and 
\begin{equation}\label{eq:bddcomp}
\lim_{|z| \to 1}\sup_{f \in K}\nu(z)|((C_{\varphi}-C_{\psi})f)^{\prime}(z)| = 0. 
\end{equation}
By Theorem \ref{thm:blochbdd} we see that (a) and (b) imply the boundedness of 
$C_{\varphi} - C_{\psi} \,:\, B^{p} \to \mathcal{B}_{\nu}$. 
Thus we will claim that $\varphi$ and $\psi$ satisfy the condition (c) in Corollary \ref{cor:bdd}. 
Since $|\varphi(z) - \psi(z)| \le 2 \rho(\varphi(z), \psi(z))$ for $z \in \mathbb{D}$, 
we see that 
\[
\nu(z) |\varphi(z) - \psi(z)| |{\varphi}^{\prime}(z)| 
\le 2 \dfrac{\nu(z)|{\varphi}^{\prime}(z)|}{1-|\varphi(z)|^{2}}\rho(\varphi(z), \psi(z)),
\]
and 
\[
\nu(z) |\varphi(z) - \psi(z)| |{\psi}^{\prime}(z)| 
\le 2 \dfrac{\nu(z)|{\psi}^{\prime}(z)|}{1-|\psi(z)|^{2}}\rho(\varphi(z), \psi(z)),
\]
and so the condition (a) shows 
\[
\lim_{|z| \to 1}\nu(z)|\varphi(z) - \psi(z)|\max\{|{\varphi}^{\prime}(z)|, |{\psi}^{\prime}(z)|\} = 0. 
\]
Moreover we also obtain that 
\begin{align*}
& \nu(z)|{\varphi}^{\prime}(z) - {\psi}^{\prime}(z)|
\\
& = 
\nu(z)\left|
\frac{(1-|\varphi(z)|^{2}){\varphi}^{\prime}(z)}{1-|\varphi(z)|^{2}}
-
\frac{(1-|\psi(z)|^{2}){\psi}^{\prime}(z)}{1-|\psi(z)|^{2}}
\right|
\\
& =
\left|
\dfrac{\nu(z){\varphi}^{\prime}(z)}{1-|\varphi(z)|^{2}} - \dfrac{\nu(z){\psi}^{\prime}(z)}{1-|\psi(z)|^{2}}
\right|
\\
&
\qquad + \nu(z)
\left|
-\dfrac{|\varphi(z)|^{2}{\varphi}^{\prime}(z)}{1-|\varphi(z)|^{2}}
+\dfrac{|\psi(z)|^{2}{\varphi}^{\prime}(z)}{1-|\varphi(z)|^{2}}
-\dfrac{|\psi(z)|^{2}{\varphi}^{\prime}(z)}{1-|\varphi(z)|^{2}}
+\dfrac{|\psi(z)|^{2}{\psi}^{\prime}(z)}{1-|\psi(z)|^{2}}
\right|
\\
& \le 
2 \left|
\dfrac{\nu(z){\varphi}^{\prime}(z)}{1-|\varphi(z)|^{2}} - \dfrac{\nu(z){\psi}^{\prime}(z)}{1-|\psi(z)|^{2}}
\right|
+ 
\left||\varphi(z)|^{2} - |\psi(z)|^{2}\right|
\dfrac{\nu(z)|{\varphi}^{\prime}(z)|}{1-|\varphi(z)|^{2}}
\\
& \le 
2 \left|
\dfrac{\nu(z){\varphi}^{\prime}(z)}{1-|\varphi(z)|^{2}} - \dfrac{\nu(z){\psi}^{\prime}(z)}{1-|\psi(z)|^{2}}
\right|
+ 
4\rho(\varphi(z), \psi(z))
\dfrac{\nu(z)|{\varphi}^{\prime}(z)|}{1-|\varphi(z)|^{2}}.
\end{align*}
Hence (a) and (b) show that $\nu(z)|{\varphi}^{\prime}(z) - {\psi}^{\prime}(z)| \to 0$ as $|z| \to 1^{-}$, 
that is $\varphi - \psi \in \mathcal{B}_{\nu , 0}$. 
By Corollary \ref{cor:bdd}, we see that 
$C_{\varphi} - C_{\psi} \,:\, B^{p} \to \mathcal{B}_{\nu ,0}$ is bounded. 

Next we will prove the equation (\ref{eq:bddcomp}). 
Fix $z \in \mathbb{D}$ and $f \in K$. 
Thus we have 
\begin{align*}
& \nu(z)|((C_{\varphi}- C_{\psi})f)^{\prime}(z)| \\
& =
\nu(z) |f^{\prime}(\varphi(z)){\varphi}^{\prime}(z) - f^{\prime}(\psi(z)){\psi}^{\prime}(z)|
\\
& =
\nu(z)
\left|
\dfrac{{\varphi}^{\prime}(z)}{1-|\varphi(z)|^{2}}(1-|\varphi(z)|^{2})f^{\prime}(\varphi(z)) 
- \dfrac{{\psi}^{\prime}(z)}{1-|\psi(z)|^{2}}(1-|\psi(z)|^{2})f^{\prime}(\psi(z))
\right|
\\
& \le 
\left|
\dfrac{\nu(z){\varphi}^{\prime}(z)}{1-|\varphi(z)|^{2}}
-
\dfrac{\nu(z){\psi}^{\prime}(z)}{1-|\psi(z)|^{2}}
\right|
(1-|\varphi(z)|^{2})|f^{\prime}(\varphi(z))|
\\
& \qquad + 
\left|
(1-|\varphi(z)|^{2})f^{\prime}(\varphi(z)) - (1-|\psi(z)|^{2})f^{\prime}(\psi(z))
\right|
\dfrac{\nu(z)|{\psi}^{\prime}(z)|}{1-|\psi(z)|^{2}}.
\end{align*}
Combining this with Lemma \ref{lem:growth} and \ref{lem:lip}, we obtain 
\[
\nu(z)|((C_{\varphi}- C_{\psi})f)^{\prime}(z)| 
\lesssim 
\left|
\dfrac{\nu(z){\varphi}^{\prime}(z)}{1-|\varphi(z)|^{2}}
-
\dfrac{\nu(z){\psi}^{\prime}(z)}{1-|\psi(z)|^{2}}
\right|
+ 
\dfrac{\nu(z)|{\psi}^{\prime}(z)|}{1-|\psi(z)|^{2}}\rho(\varphi(z), \psi(z))
\]
for any $z \in \mathbb{D}$ and $f \in K$.
Conditions (a) and (b) imply (\ref{eq:bddcomp}). 
By Lemma \ref{lem:bddcomp} we see that 
$(C_{\varphi} - C_{\psi})(K)$ is a compact subset in $\mathcal{B}_{\nu , 0}$. 
Hence $C_{\varphi} - C_{\psi} \, : \, B^{p} \to \mathcal{B}_{\nu , 0}$ is compact. 

To prove that the compactness of $C_{\varphi} - C_{\psi}$ gives conditions (a) and (b), 
we take a sequence $\{z_{n}\}$ of $\mathbb{D}$ with $|z_{n}| \to 1^{-}$ as $n \to \infty$ arbitrary. 
Put 
\[
f_{n}(z) = \dfrac{\varphi(z_{n})-z}{1-\overline{\varphi(z_{n})}z}, 
\quad \text{and}\quad 
g_{n}(z) = \left(\dfrac{\varphi(z_{n})-z}{1-\overline{\varphi(z_{n})}z}\right)^{2} 
\]
for $n \ge 1$ and $z \in \mathbb{D}$. 
Then $\{f_{n}, g_{n}\} \subset B^{p}$ and we can choose a positive constant $C$ 
which is independent of $n$, $\varphi$ and $\psi$ such that 
$\|f_{n}\|_{p} \le C$ and $\|g_{n}\|_{p} \le C$. 
Let $K_{C} = \{f \in B^{p} \,:\, \|f\|_{p} \le C\}$. 
By Lemma \ref{lem:bddcomp}, the compactness of $C_{\varphi} - C_{\psi}$ implies 
\begin{equation}\label{eq:comp}
\lim_{n \to 0}\sup_{f \in K_{C}}\nu(z_{n})|((C_{\varphi} - C_{\psi})f)^{\prime}(z_{n})| = 0.
\end{equation}
By the definition of $f_{n}$, we have 
\begin{align}\label{eq:estimate}
& \nu(z_{n})|((C_{\varphi}-C_{\psi})f_{n})^{\prime}(z_{n})| \notag \\
& = \nu(z_{n})
\left|
\dfrac{-{\varphi}^{\prime}(z_{n})}{1-|\varphi(z_{n})|^{2}} 
+ 
\dfrac{{\psi}^{\prime}(z_{n})(1-|\varphi(z_{n})|^{2})}{(1-\overline{\varphi(z_{n})}\psi(z_{n}))^{2}}
\right| 
\notag \\
& \ge 
\left|
\dfrac{\nu(z_{n})|{\varphi}^{\prime}(z_{n})|}{1-|\varphi(z_{n})|^{2}}
-
\dfrac{\nu(z_{n})|{\psi}^{\prime}(z_{n})|}{1-|{\psi}(z_{n})|^{2}}(1-\rho(\varphi(z_{n}),\psi(z_{n})))^{2})
\right|.
\end{align}
Hence (\ref{eq:comp}) gives 
\begin{equation}\label{eq:001}
\lim_{n \to \infty}
\left|
\dfrac{\nu(z_{n})|{\varphi}^{\prime}(z_{n})|}{1-|\varphi(z_{n})|^{2}}
-
\dfrac{\nu(z_{n})|{\psi}^{\prime}(z_{n})|}{1-|{\psi}(z_{n})|^{2}}(1-\rho(\varphi(z_{n}),\psi(z_{n})))^{2})
\right|= 0.
\end{equation}
Since $g_{n}^{\prime}(\varphi(z_{n})) = 0$, we have 
\begin{align*}
& \nu(z_{n})|((C_{\varphi}-C_{\psi})g_{n})^{\prime}(z_{n})|
\\
& = 
2\nu(z_{n}) 
\dfrac{|\varphi(z_{n}) - \psi(z_{n})||{\psi}^{\prime}(z_{n})|}{|1-\overline{\varphi(z_{n})}\psi(z_{n})|^{3}}(1-|\varphi(z_{n})|^{2})
\\
& = 
2 \dfrac{\nu(z_{n})|{\psi}^{\prime}(z_{n})|}{1-|\psi(z_{n})|^{2}}\rho(\varphi(z_{n}), \psi(z_{n}))(1-\rho(\varphi(z_{n}), \psi(z_{n}))^{2}). 
\end{align*}
The equation (\ref{eq:comp}) also gives 
\begin{equation}\label{eq:002}
\lim_{n \to \infty}
\dfrac{\nu(z_{n})|{\psi}^{\prime}(z_{n})|}{1-|\psi(z_{n})|^{2}}\rho(\varphi(z_{n}), \psi(z_{n}))(1-\rho(\varphi(z_{n}), \psi(z_{n}))^{2})
=0.
\end{equation}
By replacing the role of $\varphi$ and $\psi$ in definitions $f_{n}$ and $g_{n}$, 
we also see that $\varphi$ and $\psi$ satisfy 
\begin{equation}\label{eq:003}
\lim_{n \to \infty}
\left|
\dfrac{\nu(z_{n})|{\psi}^{\prime}(z_{n})|}{1-|\psi(z_{n})|^{2}}
-
\dfrac{\nu(z_{n})|{\varphi}^{\prime}(z_{n})|}{1-|{\varphi}(z_{n})|^{2}}(1-\rho(\varphi(z_{n}),\psi(z_{n})))^{2})
\right|= 0, 
\end{equation}
and
\begin{equation}\label{eq:004}
\lim_{n \to \infty}
\dfrac{\nu(z_{n})|{\varphi}^{\prime}(z_{n})|}{1-|\varphi(z_{n})|^{2}}\rho(\varphi(z_{n}), \psi(z_{n}))(1-\rho(\varphi(z_{n}), \psi(z_{n}))^{2})
=0.
\end{equation}
Now we assume that 
\[
\lim_{n \to \infty}\dfrac{\nu(z_{n})|{\varphi}^{\prime}(z_{n})|}{1-|\varphi(z_{n})|^{2}} \rho(\varphi(z_{n}), \psi(z_{n})) \neq 0. 
\]
Then (\ref{eq:004}) indicates that 
$1-\rho(\varphi(z_{n}),\psi(z_{n}))^{2} \to 0$ as $n \to \infty$. 
By (\ref{eq:001}) we obtain 
\[
\lim_{n \to \infty}\dfrac{\nu(z_{n})|{\varphi}^{\prime}(z_{n})|}{1-|\varphi(z_{n})|^{2}} = 0.
\]
Since $\rho(\varphi(z_{n}), \psi(z_{n})) < 1$, this claim contradicts our assumption. 
By the same argument with (\ref{eq:002}) and (\ref{eq:003}), 
we also have that 
\begin{equation}\label{eq:lim(a)}
\lim_{n \to \infty}\dfrac{\nu(z_{n})|{\psi}^{\prime}(z_{n})|}{1-|\psi(z_{n})|^{2}} \rho(\varphi(z_{n}), \psi(z_{n})) = 0.
\end{equation}
Since $\{z_{n}\} \subset \mathbb{D}$ with $|z_{n}| \to 1$ as $n \to \infty$ was arbitrary, 
these imply the condition (a) holds. 
Furthemore, the estimate (\ref{eq:estimate}) gives 
\begin{align*}
& \nu(z_{n})|((C_{\varphi}-C_{\psi})f_{n})^{\prime}(z_{n})| 
\\
& \ge 
\left|
\dfrac{\nu(z_{n})|{\varphi}^{\prime}(z_{n})|}{1-|\varphi(z_{n})|^{2}}
-
\dfrac{\nu(z_{n})|{\psi}^{\prime}(z_{n})|}{1-|{\psi}(z_{n})|^{2}}
\right|
-
\dfrac{\nu(z_{n})|{\psi}^{\prime}(z_{n})|}{1-|{\psi}(z_{n})|^{2}}
\rho(\varphi(z_{n}),\psi(z_{n})). 
\end{align*}
By (\ref{eq:comp}) and (\ref{eq:lim(a)}), we obtain 
\[
\lim_{n \to \infty}
\left|
\dfrac{\nu(z_{n})|{\varphi}^{\prime}(z_{n})|}{1-|\varphi(z_{n})|^{2}}
-
\dfrac{\nu(z_{n})|{\psi}^{\prime}(z_{n})|}{1-|{\psi}(z_{n})|^{2}}
\right| =0, 
\]
and so this indicates the condition (b). 
\end{proof}

\begin{ack}
This research is partly supported by JSPS KAKENHI Grants-in-Aid 
for Scientific Research (C), Grant Number 17K05282. 
Partial work on this paper was done while the second author
visited the Department of Mathematics, Central University of Jammu, Jammu.  
He wishes to thank Central University of Jammu for hosting his visit.
The first author is thankful to NBHM(DAE)(India) for the project grant No. 02011/30/2017/R\&{D} II/12565.
\end{ack}

\end{document}